\documentclass[12pt]{article}%
\usepackage{amsmath}
\usepackage{amssymb}
\usepackage{amsfonts}
\usepackage{mitpress}
\usepackage{color}
\usepackage{graphicx}%
\providecommand{\U}[1]{\protect\rule{.1in}{.1in}}
\newtheorem{theorem}{Theorem}

\newtheorem{definition}[theorem]{Definition}

\newtheorem{lemma}[theorem]{Lemma}

\newtheorem{remark}[theorem]{Remark}

\newenvironment{proof}[1][Proof]{\noindent\textbf{#1.} }{\ \rule{0.5em}{0.5em}}
\newdimen\dummy
\dummy=\oddsidemargin
\usepackage{mathrsfs}

\newcommand{\xdois}{\mathbb{X}^2}
\newcommand{\vp}{\mathbb{V}^p}

\newcommand{\deln}{\partial_{\overrightarrow{n}}}

\newcommand{\gl}{\left\langle} 
\newcommand{\gr}{\right\rangle}

\newcommand{\R}{\ensuremath{\mathbb{R}}}
\newcommand{\N}{\ensuremath{\mathbb{N}}}

\usepackage{tikz-cd}

\begin{document}

\title{Continuity of Selected Pullback Attractors}
\author{Rodrigo Samprogna$^{1}$, Jacson Simsen$^{2}$
\\$^{1}${\small Instituto de Ci\^{e}ncia e Tecnologia, Universidade Federal de
Alfenas, Campus de Po\c{c}os de Caldas,  Rod. Jos\'{e} Aur\'{e}lio Vilela, n. 11.999, Cidade universitária, 37715-400,
Po\c{c}os de Caldas - MG - Brazil.}
\\$^{2}${\small Instituto de
Matem\'{a}tica e Computa\c c\~ao, Universidade Federal de
Itajub\'{a}, Av. BPS n. 1303, Bairro Pinheirinho, 37500-903,
Itajub\'{a} - MG - Brazil.}
\\$^{1}${\small E-mail: rodrigo.samprogna@unifal-mg.edu.br}
\\$^{2}${\small E.mail: jacson@unifei.edu.br}}
\date{}
\maketitle

\vspace{2cm}
\begin{abstract} In this work we obtain theoretical results on continuity of selected pullback attractors and we apply them to reaction diffusion equations with dynamical boundary conditions. 
\end{abstract}



\section{Introduction}\label{Introduction} The study of the asymptotic behavior of infinite dimensional evolution problems is connected with the existence of  attractors for semigroups (or processes for nonautonomous problems) associated with partial differential equations. Sometimes the uniqueness of solutions of the Initial Value Problem fails or it is not known to hold. In this case one have to use the theory of multivalued semigroups (or multivalued processes for nonautonomous problems).

For some problems it is not possible to guarantee the existence of global attractors or pullback attractors. For this reason, some researchers had introduced different weaker concepts of attractors. Simsen-Gentile in \cite{SG} had introduced the notion of $\varphi-$attractor for a multivalued semigroup defined by a generalized semiflow. Kapustyan-Pankov-Valero in \cite{KPV,KP} considered 3D-B\'enard systems which is a model in hydrodynamics and describes
the behaviour of the velocity, the pressure and the temperature of an incompressible fluid.These systems, in general, do not have a global attractor because of the lack of good dissipativity estimates for all weak solutions. However, they proved the existence of a $\varphi-$attractor for the 3D-B\'enard systems. See also \cite{KV,KMV}. Caraballo-Kloeden-Mar\'in--Rubio in \cite{CKM03} introduced the concept of a weak pullback attractor and Samprogna-Simsen in \cite{spa} introduced the concept of Selected pullback attractors. 

In this paper we develop some abstract results and apply them to a $p-$Laplacian evolution equation without guarantee of uniqueness of solution. In the Section~\ref{Selected Pullback Attractor}, for the completeness of the work, we revise some concepts and results from \cite{spa}. Section~\ref{Continuity of Selected Pullback Attractors} present new theoretical results on continuity of selected pullback attractors. The last section is devoted to give an application to a $p-$Laplacian reaction-diffusion equation with dynamic boundary conditions.

\section{Selected Pullback Attractor}\label{Selected Pullback Attractor}

Let $(X,d)$ be a complete metric space. For $x\in X$, $A,B\subset X$ and $\epsilon>0$ we define
\begin{align*}
& d(x,A):=\inf_{a\in A}\{d(x,a)\}; \\
& dist(A,B):=\sup_{a\in A}\inf_{b\in B}\{d(a,b)\}; \\
& dist_H(A,B):=\max\{dist(A,B),dist(B,A)\}; \\
& \mathcal{O}_{\epsilon}(A):=\{z\in X;d(z,A)<\epsilon\}.
\end{align*} 
Denote by $\mathcal{P}(X)$, $\mathcal{B}(X)$ and $\mathcal{K}(X)$ the nonempty, nonempty and bounded and nonempty and compact subsets of $X$, respectively.

\begin{definition}[Generalized Process]\label{processo generalizado} A \textbf{generalized process}  $\mathscr{G}=\{\mathscr{G}(\tau)\}_{\tau\in\R}$ in $X$ is a family of sets $\mathscr{G}(\tau)$ consisting of functions $\varphi:[\tau,\infty)\rightarrow X$, called solutions, satisfying the following conditions:
\begin{enumerate}
\item[(C1)] For each $\tau\in\R$ and $z\in X$ there exists at least one $\varphi\in\mathscr{G}$ with $\varphi(\tau)=z$;
\item[(C2)] If $\varphi\in\mathscr{G}(\tau)$ and $s\geq 0$, then $\varphi^{+s}\in\mathscr{G}(\tau+s)$, where $\varphi^{+s}:=\varphi_{|[\tau+s,\infty)}$;
\item[(C3)] If $\{\varphi_j\}_{j\in\N}\subset\mathscr{G}(\tau)$ and $\varphi_j(\tau)\rightarrow z$, there is a subsequence $\{\varphi_{j_k}\}_{k\in\N}$ of $\{\varphi_j\}_{j\in\N}$ and $\varphi\in\mathscr{G}(\tau)$ with $\varphi(\tau)=z$ and such that $\varphi_{j_k}(t)\rightarrow\varphi(t)$ when $k\rightarrow\infty$, for each $t\geq \tau$.
\end{enumerate}
\end{definition}

Whether $\mathscr{G}$ is formed by continuous functions we call this process a continuous process.

\begin{definition}
A generalized process $\mathscr{G}=\{\mathscr{G}(\tau)\}_{\tau\in\R}$ is said to be \textbf{locally uniformly upper semicontinuous (LUUS)} if it satisfies the following condition:
\begin{enumerate}
\item[(C4)] If $\{\varphi_j\}_{j\in\N}$ is a sequence such that $\varphi_j\in  \mathscr{G}(\tau)$ and $\varphi_j(\tau)\to z$, then there is $\varphi\in \mathscr{G}(\tau)$ with $\varphi(\tau)=z$ and subsequence $\{\varphi_{j_k}\}_{k\in\N}$ such that $\varphi_{j_k}\rightarrow \varphi$ uniformly on compact subsets of $[\tau,+\infty)$ when $k\to\infty$.
\end{enumerate}
\end{definition}

\begin{definition}
 We say that a generalized process $\mathscr{G}=\{\mathscr{G}(\tau)\}_{\tau\in\R}$ is \textbf{exact (or strict)} if it satisfies the following condition:
\begin{enumerate}
\item[(C5)] (Concatenation) Let $\varphi\in\mathscr{G}(\tau)$ and $\psi\in\mathscr{G}(r)$ such that $\varphi(s)=\psi(s)$ for some $s\geq r\geq\tau$. If $\theta$ is defined by 
\begin{equation*}
\theta(t):=\left\{
\begin{array}{ll}
\varphi(t), & t\in[\tau,s], \cr
\psi(t), & t\in(s,\infty),
\end{array}
\right.
\end{equation*}
then $\theta\in \mathscr{G}(\tau)$.
\end{enumerate}

\end{definition}

\begin{definition}
We say that there exists a \textbf{complete orbit} through $x\in X$ at $\tau\in\mathbb{R}$ if there is a map $\psi:\mathbb{R}\to X$ with $\psi(\tau)=x$ and for all $s\in\mathbb{R}$, $\psi_{|[\tau+s,\infty)}\in\mathscr{G}(\tau+s)$.
\end{definition}

We refer the reader to \cite{SimCap,SimsenValero} for more details on generalized process theory.

\begin{definition}
Let $\mathcal{A}=\{\mathcal{A}(t)\}_{t\in\mathbb{R}}$ be a family of sets. We say that $\mathcal{A}$ \textbf{select pullback attracts} an element $x\in X$ at time $t\in\mathbb{R}$ if given $\varepsilon>0$ there is $\tau_0\leq t$ such that for all $\tau\leq\tau_0$ there is $\varphi_{\tau}\in\mathscr{G}(\tau)$ with $\varphi_{\tau}(\tau)=x$ and 
$$\varphi_{\tau}(t)\in\mathcal{O}_{\varepsilon}(\mathcal{A}(t)).$$    
\end{definition}

\begin{definition}
Let $\mathcal{A}=\{\mathcal{A}(t)\}_{t\in\mathbb{R}}$ be a family of subsets of $X$. We say that $\mathcal{A}$ is \textbf{quasi-invariant} if for each $z\in \mathcal{A}(\tau)$ for some $\tau \in\mathbb{R}$, there exists a complete orbit $\psi$ through $z$ at $\tau$ with $\psi(t)\in \mathcal{A}(t)$ for all $t\in\mathbb{R}$.     
\end{definition}

\begin{remark}
Quasi-invariance is named as weak invariance in \cite{CKM03}.
\end{remark}

\begin{definition}\label{Ku}
We say that a generalized process $\mathscr{G}$ has the \textbf{uniform selected K-property} if there exist a family of compact sets $\{K(t)\}_{t\in\mathbb{R}}$ such that given $x\in X$ for each $\tau\in\mathbb{R}$ there is $\varphi_{\tau}\in\mathscr{G}(\tau)$ with $\varphi_{\tau}(\tau)=x$, and this family of solutions $\{\varphi_{\tau}\}_{\tau\in\mathbb{R}}$ has the following property: for each $t\in\mathbb{R}$ there is $\tau_0\leq t$ and sequence $\{\varepsilon_{\tau}(t)\}_{\tau\leq\tau_0}$ such that 
$$\varphi_{\tau}(t)\in\mathcal{O}_{\varepsilon_{\tau}}(K(t)), \ \tau\leq \tau_0,$$
with $\varepsilon_{\tau}\to 0$ as $\tau\to -\infty$. 
\end{definition}

\begin{theorem}\label{teo25}[\cite{spa},Theorem 25] 
If the generalized process $\mathscr{G}$ possess the uniform selected $K$-property, then it have a quasi-invariant global select pullback attractor $\mathcal{A}=\{\mathcal{A}(t)\}_{t\in\mathbb{R}}$.    
\end{theorem}

\section{Continuity of Selected Pullback Attractors}\label{Continuity of Selected Pullback Attractors}

In this section we present abstract new results.

\begin{theorem}\label{asdf}
Let $\mathscr{G}$ a continuous and LUUS generalized process with the uniform selected K-property and the family $\{\mathcal{A}(t)\}_{t\in\R}$ is the global selected pullback attractor. Then the setvalued mapping $t\to \mathcal{A}(t)$ is continuous, i.e., 
$$\lim_{s\to t}dist_H(\mathcal{A}(s),\mathcal{A}(t))=0, \ \forall \ t\in \R.$$  
\end{theorem}

\begin{proof}
Fixed $t\in\R$, firstly, consider the limit $dist(\mathcal{A}(t),\mathcal{A}(s))$ as $s\to t$. Suppose that $\displaystyle\lim_{s\to t}dist(\mathcal{A}(t),\mathcal{A}(s))$ is not zero, then there would exist $\delta>0$ and a sequence $s_n\to t$ such that
\begin{equation}\label{contrad}
\delta \leq dist(\mathcal{A}(t),\mathcal{A}(s_n)), \quad \forall \ n\in\N.
\end{equation}

Since $\mathcal{A}(t)$ is compact, for each $n\in \N$ there exist $a_n\in \mathcal{A}(t)$ with 
$$dist(\mathcal{A}(t),\mathcal{A}(s_n))=dist(a_n,\mathcal{A}(s_n))\leq dist(a_n,a_{s_n}))$$
for all $a_{s_n}\in\mathcal{A}(s_n)$.    

From Theorem \ref{teo25}, for each $a_n\in\mathcal{A}(t)$ there exist a complete trajectory $\varphi_n\in\mathscr{G}(t)$ such that $\varphi_n(s)\in\mathcal{A}(s)$ for all $s\in\R$ with $\varphi_n(t)=a_n$. As $\mathcal{A}(t)$ is a compact set there is $a\in\mathcal{A}(t)$ such that, up to subsequence, $a_n\to a$.    

From the LUUS property, there exist $\varphi_t\in\mathscr{G}(t)$ such that $\varphi_n\to \varphi_t$ uniformly on compact subsets of $[t, +\infty)$.

However, we have that $b_n:=\varphi_n(t-1)\in\mathcal{A}(t-1)$ for each $n\in\N$, then, analogous to the above, there exist $b\in\mathcal{A}(t-1)$ and $\varphi_{t-1}\in\mathscr{G}(t-1)$ such that, up to subsequence, $b_n\to b$ and $\varphi_{t-1}(t-1)=b$ with  $\varphi_n\to \varphi_{t-1}$ uniformly on compact subsets of $[t-1, +\infty)$. Moreover, $\varphi_{t-1}(s)=\varphi_t(s)$ for each $s\in[t,+\infty)$.

Given $s\in[t-1,\infty)$, define $\varphi(s)$ as the common value of $\varphi_{t-1}(s)$ and $\varphi_t(s)$ if $s\in[t,\infty)$ or just $\varphi_{t-1}(s)$ if $s\in[t-1,t)$. Then $\varphi_n\to\varphi$ uniformly on compact subsets of $[t-1,\infty)$. 

Note that, for $n$ large enough we have that $s_n\in[t-1,\infty)$, in this case, from the continuity of the trajectories and the uniform convergence of $\varphi_n\to \varphi$ on compact subsets of $[t-1,+\infty)$, we have that 
\begin{align*}
dist(\mathcal{A}(t),\mathcal{A}(s_n))&\leq d(\varphi_n(t),\varphi_n(s_n)) \\
&\leq d(\varphi_n(t),\varphi(t))+d(\varphi(t),\varphi(s_n))+d(\varphi(s_n),\varphi_n(s_n))\to 0.
\end{align*}
contradicting \eqref{contrad}. Therefore, $\displaystyle\lim_{s\to t}dist(\mathcal{A}(t),\mathcal{A}(s))=0.$

On the other hand, now let us show $\displaystyle\lim_{s\to t}dist(\mathcal{A}(s),\mathcal{A}(t))=0$. Suppose that it does not hold, then there would exist $\delta>0$ and a sequence $s_n\to t$ such that
\begin{equation}\label{contrad2}
\delta \leq dist(\mathcal{A}(s_n),\mathcal{A}(t)) \quad \forall \ n\in\N.
\end{equation}

Since $\mathcal{A}(s_n)$ is compact, there exist $a_n\in\mathcal{A}(s_n)$ such that 
$$dist(\mathcal{A}(s_n),\mathcal{A}(t))= dist(a_n,\mathcal{A}(t))\leq dist(a_n,a),$$
for all $a\in\mathcal{A}(t)$.

From Theorem \ref{teo25}, for each $a_n$ there exist a complete trajectory $\varphi_n$ with $\varphi_n(s_n)=a_n$ and $\varphi_n(s)\in\mathcal{A}(s)$ for all $s\in\R$. 

As $s_n\to t$ there exist $n_0\in\N$ such that $s_n\in[t-1,\infty)$ for all $n\geq n_0$. Consider a subsequence starting on $s_{n_0}$, which we do not relabel. 

  We have that 
  $$\varphi_n(t-1)\in\mathcal{A}(t-1),$$
  and, provided that $\mathcal{A}(t-1)$ is compact, there is $a_0\in\mathcal{A}(t-1)$ such that, up to subsequence, $\varphi_n(t-1)\to a_0$ as $n\to +\infty$. 
  
  From the LUUS property there exist $\varphi\in\mathscr{G}(t-1)$ with $\varphi_n\to\varphi$ uniformly on compact subsets of $[t-1,+\infty)$.

From the continuity of the trajectories and the uniform convergence of $\varphi_n\to \varphi$ on compact subsets of $[t-1,+\infty)$, we have that 
\begin{align*}
dist(\mathcal{A}(s_n),\mathcal{A}(t))&\leq d(\varphi_n(s_n),\varphi_n(t)) \\
&\leq d(\varphi_n(s_n),\varphi(s_n))+d(\varphi(s_n),\varphi(t))+d(\varphi(t),\varphi_n(t))\to 0.
\end{align*}
contradicting \eqref{contrad2}.
\end{proof}

\begin{remark}
One version of the above theorem can be found in Proposition 11 of \cite{CKM03}, where the authors have used an exact and locally uniform upper semicontinuity in $t$ for a setvalued process (or multivalued process). Note that, in our case the setvalued process generated by the generalized process $\mathscr{G}$ on the above theorem do not necessary be an exact setvalued process, and the uniformly upper semicontinuity can be obtained by LUUS property.   
\end{remark}

\begin{theorem}\label{qwer} Let $\mathcal{A}$ be a forward compact selected pullback attractor, i.e., 
$$\overline{\cup_{s\geq t}\mathcal{A}(s)}$$
is a compact set for all $t\in\R$. Then, we have that
\begin{equation}\label{dist}
lim_{t\to\infty}dist(\mathcal{A}(t),\mathcal{A}(\infty))=0,
\end{equation}
where $\mathcal{A}(\infty):=\cap_{t\in\R}\overline{\cup_{r\geq t}\mathcal{A}(r)}$. 
\end{theorem}

\begin{proof}
Note that, since $\mathcal{A}$ is forward compact the set $\mathcal{A}(\infty)$ is a compact set. 

Suppose that \eqref{dist} is not true, then there are $\delta>0$ and a real sequence $\{t_n\}_{n\in\N}$ with $0<t_n\nearrow+\infty$ such that $dist(\mathcal{A}(t_n),\mathcal{A}(\infty))\geq \delta$ for all $n\in\N$. 

Thus, for each $n\in \N$ there is $x_n\in\mathcal{A}(t_n)$ such that 
\begin{equation}\label{ABS}
dist(x_n,\mathcal{A}(\infty))\geq\delta.
\end{equation}
Note that, for some $t_0\geq 0$, we have $\{x_n\}_{n\in\N}\subset\overline{\cup_{s\geq t_0}\mathcal{A}(s)}$ for $t_n\geq t_0$.

As $\mathcal{A}$ is a forward compact family we have that the sequence $\{x_n\}_{n\in\N}$ has got a convergent subsequence, which we do not relabel, and let $x\in X$ a limit of this subsequence. 

From the definition we have $x\in \mathcal{A}(\infty)$, and then \eqref{ABS} is a contradiction. 
\end{proof}

\begin{theorem} Let $\mathcal{A}$ be a backward compact selected pullback attractor, i.e.,
$$\overline{\cup_{s\leq t}\mathcal{A}(s)}$$
is a compact set for all $t\in\R$. then, we have 
\begin{equation}\label{dist2}
\lim_{t\to -\infty}dist(\mathcal{A}(t),\mathcal{A}(-\infty))=0,
\end{equation}
where $\mathcal{A}(-\infty):=\cap_{t\in\R}\overline{\cup_{r\leq t}\mathcal{A}(r)}.$
\end{theorem}

\begin{proof}
Note that, since $\mathcal{A}$ is backward compact the set $\mathcal{A}(-\infty)$ is a compact set. 

Suppose that \eqref{dist2} is not true, then there are $\delta>0$ and a real sequence $\{t_n\}_{n\in\N}$ with $0>t_n\searrow-\infty$ such that $dist(\mathcal{A}(t_n),\mathcal{A}(-\infty))\geq \delta$ for all $n\in\N$. 

Thus, for each $n\in \N$ there is $x_n\in\mathcal{A}(t_n)$ such that 
\begin{equation}\label{ABS2}
dist(x_n,\mathcal{A}(-\infty))\geq\delta.
\end{equation}
Note that, for some $t_0\leq 0$, we have $\{x_n\}_{n\in\N}\subset\overline{\cup_{s\leq t_0}\mathcal{A}(s)}$ for $t_n\leq t_0$.

As $\mathcal{A}$ is a backward compact family we have that the sequence $\{x_n\}_{n\in\N}$ has a convergent subsequence, which we do not relabel, and let $x\in X$ a limit of this subsequence. 

From the definition we have $x\in \mathcal{A}(-\infty)$, and then \eqref{ABS2} is a contradiction.   
\end{proof}

\begin{definition} Let $\mathcal{A}=\{\mathcal{A}(t)\}$ be a quasi-invariant global selected pullback attractor of a generalized process $\mathscr{G}$. We say that the generalized process $\mathscr{G}$ is \textbf{asymptotically autonomous for $\mathcal{A}$} if for each sequence $\{x_{\tau}\}\subset X$ such that $x_{\tau}\in\mathcal{A}(\tau)$ and $x_{\tau}\to x_0$ as $\tau\to\infty$, let $\varphi_{\tau}$ a complete trajectory with $\varphi_{\tau}(\tau)=x_{\tau}$ and $\varphi_{\tau}(s)\in\mathcal{A}(s)$ for all $s\in\R$, we have that there exist a solution $\varphi$ of an autonomous problem such that 
$$\varphi_{\tau}(t+\tau)\to\varphi(t)\;\textrm{as}\; \tau\to\infty\; \textrm{and}\; \varphi(0)=x_0.$$
\end{definition}

\begin{theorem}\label{TeoX}
Let $\mathscr{G}$ be an asymptotically autonomous for a selected pullback attractor $\mathcal{A}$ and let $\mathcal{A}_{\infty}$ be the global attractor of the corresponding autonomous problem. If the selected pullback attractor $\mathcal{A}$ is forward compact, then 
\begin{equation}\label{conv1}
\displaystyle \lim_{t\to\infty}dist(\mathcal{A}(t),\mathcal{A}_{\infty})= 0.
\end{equation} 
\end{theorem}

\begin{proof} Let $\mathcal{A}$ forward compact and suppose \eqref{conv1} is not true. Then there would exist an $\varepsilon_0>0$ and a real sequence $\{\tau_n\}$ with $0<\tau_n\nearrow +\infty$ such that $dist(\mathcal{A}(\tau_n),\mathcal{A}_{\infty})\geq 3\varepsilon_0$ for all $n\in\N$. Since the sets $\mathcal{A}(\tau_n)$ are compact, there exist $a_n\in\mathcal{A}(\tau_n)$ such that
\begin{equation}\label{XXX}
dist(a_n,\mathcal{A}_{\infty})=dist(\mathcal{A}(\tau_n),\mathcal{A}_{\infty})\geq 3\varepsilon_0,
\end{equation}
for each $n\in\N$. 

Let $C=\overline{\cup_{s\geq 0}\mathcal{A}(s)}$ a compact set, there is $n_0>0$ such that 
$$dist(G(\tau_{n_0},C),\mathcal{A}_{\infty})\leq \varepsilon_0.$$
with $G(t,\cdot)$ is the multivalued semigroup of the corresponding autonomous problem.  

Note that the sequence $0<\tau_n-\tau_{n_0}\nearrow \infty$ for all $n\geq n_0$, and there are a sequence of complete orbits $\varphi_{\tau_n-\tau_{n_0}}$ and a sequence $\{b_n\}_{n\geq n_0}$ with $\varphi_{\tau_n-\tau_{n_0}}(\tau_n-\tau_{n_0})=b_n\in\mathcal{A}(\tau_n-\tau_{n_0})$ and 
$$a_n= \varphi_{\tau_n-\tau_{n_0}}(\tau_n).$$

We have that $\{b_n\}\subset C$, then there is $b$ such that, up to a subsequence, $b_n\to b$

We have that,
\begin{align*}
a_n=\varphi_{\tau_n-\tau_{n_0}}(\tau_{n_0}+\tau_n-\tau_{n_0})\to  \varphi(\tau_{n_0})   
\end{align*}
with $\varphi(0)=b$, and then $\varphi(\tau_{n_0})\in G(\tau_{n_0},C).$

Therefore, for $n$ large enough, we have  
\begin{align*}
dist(a_n,\mathcal{A}_{\infty})\leq \|a_n-\varphi(\tau_{n_0})\|+dist(\varphi(\tau_{n_0}),\mathcal{A}_{\infty})\leq 2\varepsilon_0
\end{align*} 
which contradicts \eqref{XXX}.
\end{proof}

\begin{theorem}
Let $\mathscr{G}$ be a LUUS process composed by continuous functions and asymptotically autonomous for $\mathcal{A}$ with $\mathcal{A}_{\infty}$ be the global attractor of the corresponding autonomous problem. If  
\begin{equation}\label{conv2}
\displaystyle \lim_{t\to\infty}dist(\mathcal{A}(t),\mathcal{A}_{\infty})\to 0,
\end{equation} 
then  the selected pullback attractor $\mathcal{A}$ is forward compact.
\end{theorem}

\begin{proof}
Suppose that \eqref{conv2} is true. For fixed $t\in\R$ let $\{x_n\}\subset \cup_{t\leq r}\mathcal{A}(r)$. For each $n\in\N$ there is $r_n$ such that $x_n\in A(r_n)$. Thus we have two cases:

\

\noindent\textbf{Case 1:} $r_0:=sup_nr_n<\infty$.

In this case there is a sequence $\{\varphi_n\}$ with $\varphi_n(r_n)=x_n$ and $\varphi_n(t)\in\mathcal{A}(t)$, as the set $\mathcal{A}(t)$ is compact there is $b\in X$ such that, up to a subsequence, $\varphi_n(t)\to b$. From the LUUS property of $\mathscr{G}$ there exist $\varphi\in\mathscr{G}(t)$ such that $\varphi_n\to\varphi$ on compact subsets of $[t,\infty)$.

As $\{r_n\}\subset [t,r_0]$ there is $r'\in [t,r_0]$ such that, up to a subsequence, $r_n\to r'$. Given $\varepsilon>0$ and from the continuity of the solutions and uniform convergence of $\varphi_n$ in $[t,r_0]$, for $n$ large enough, we get 
\begin{align*}
\|\varphi_n(r_n)-\varphi(r')\|\leq \|\varphi_n(r_n)-\varphi(r_n)\|+\|\varphi(r_n)-\varphi(r')\|<\varepsilon.
\end{align*}

\

\noindent\textbf{Case 2:} $\sup_nr_n=\infty$.

In this case, up to a subsequence, we may assume $r_n\nearrow \infty$. We have
\begin{align*}
dist(x_n,\mathcal{A}_{\infty})\leq dist(\mathcal{A}(r_n),\mathcal{A}_{\infty})\to 0.
\end{align*}
We can choose $y_n\in\mathcal{A}_{\infty}$ such that 
$$d(x_n,y_n)\leq dist(x_n,\mathcal{A}_{\infty})+\frac{1}{n}.$$
There is $y\in\mathcal{A}_{\infty}$ such that, up to a subsequence, $y_n\to y$, which implies that $x_n\to y$.
\end{proof}

\begin{definition} Let $\mathcal{A}=\{\mathcal{A}(t)\}$ be a quasi-invariant global selected pullback attractor of a generalized process $\mathscr{G}$. We say that the generalized process $\mathscr{G}$ is \textbf{asymptotically backward autonomous} for $\mathcal{A}$ if for each sequence $\{x_{\tau}\}\subset X$ such that $x_{\tau}\in\mathcal{A}(\tau)$ and $x_{\tau}\to x_0$ when $\tau\to-\infty$, let $\varphi_{\tau}$ a complete trajectory with $\varphi_{\tau}(\tau)=x_{\tau}$ and $\varphi_{\tau}(s)\in\mathcal{A}(s)$ for all $s\in\R$, we have that there exist a solution $\varphi$ of an autonomous problem such that 
$$\varphi_{\tau}(t+\tau)\to\varphi(t)$$
with $\tau\to-\infty$ and $\varphi(0)=x_0$. 
\end{definition}

\begin{theorem}
Let $\mathscr{G}$ be an asymptotically backward autonomous process for a selected pullback attractor $\mathcal{A}$ and let $\mathcal{A}_{\infty}$ be the global attractor of the corresponding autonomous problem. If the selected pullback attractor is backwards compact, then
\begin{equation}\label{conv11}
\lim_{t\to -\infty}dist(\mathcal{A}(t),\mathcal{A}_{\infty})=0.
\end{equation}
\end{theorem}

\begin{proof}

Let $\mathcal{A}$ backward compact and suppose \eqref{conv11} is not true. Then, there would exist an $\varepsilon_0>0$ and a real sequence $\{\tau_n\}$ with $0<\tau_n\nearrow +\infty$ such that $dist(\mathcal{A}(-\tau_n),\mathcal{A}_{\infty})\geq 3\varepsilon_0$ for all $n\in\N$. Since the sets $\mathcal{A}(-\tau_n)$ are compact, there exists $a_n\in\mathcal{A}(-\tau_n)$ such that 
\begin{equation}\label{XX}
dist(a_n,\mathcal{A}_{\infty})=dist(\mathcal{A}(-\tau_n),\mathcal{A}_{\infty})\geq 3\varepsilon_0,
\end{equation}
for each $n\in\N$. 

Let $C=\overline{\cup_{s\leq 0}\mathcal{A}(s)}$ a compact set, there is $T_0>0$ such that 
$$dist(G(T_0,C),\mathcal{A}_{\infty})\leq \varepsilon_0.$$
with $G(t,\cdot)$ is the multivalued semigroup of the corresponding autonomous problem.  

Note that the sequence $0>-\tau_n-T_0\searrow -\infty$, and there are a sequence of complete orbits $\{\varphi_{-\tau_n-T_0}\}_{n\in\N}$ and a sequence $\{b_n\}_{n\in\N}$ with $\varphi_{-\tau_n-T_0}(-\tau_n-T_0)=b_n\in\mathcal{A}(-\tau_n-T_0)$ and 
$$a_n= \varphi_{-\tau_n-T_0}(-\tau_n).$$

We have that $\{b_n\}\subset C$, then there is $b\in C$ such that, up to a subsequence, $b_n\to b.$

We have then,
\begin{align*}
a_n=\varphi_{-\tau_n-T_0}(T_0-\tau_n-T_0)\to  \varphi(T_0)   
\end{align*}
with $\varphi(0)=b$, and then $\varphi(T_0)\in G(T_0,C).$

Therefore, for $n$ large enough, we have  
\begin{align*}
dist(a_n,\mathcal{A}_{\infty})\leq \|a_n-\varphi(T_0)\|+dist(\varphi(T_0),\mathcal{A}_{\infty})\leq 2\varepsilon_0
\end{align*} 
which contradicts \eqref{XX}.
\end{proof}

\begin{theorem}
Let $\mathscr{G}$ be a LUUS process composed by continuous functions and asympto\-ti\-cally backwards autonomous for $\mathcal{A}$ with $\mathcal{A}_{\infty}$ be the global attractor of the corresponding autonomous problem. If  
\begin{equation}\label{conv22}
\displaystyle \lim_{t\to-\infty}dist(\mathcal{A}(t),\mathcal{A}_{\infty})\to 0,
\end{equation} 
then  the selected pullback attractor $\mathcal{A}$ is backwards compact.
\end{theorem}

\begin{proof}
Suppose that \eqref{conv22} is true. For fixed $t\in\R$ let $\{x_n\}\subset \cup_{r\leq t}\mathcal{A}(r)$. For each $n\in\N$ there is $r_n$ such that $x_n\in A(r_n)$. Thus we have two cases:

\

\noindent\textbf{Case 1:} $r_0:=\inf_nr_n>-\infty$.

In this case there is a sequence $\{\varphi_n\}$ with $\varphi_n(r_n)=x_n$ and $\varphi_n(r_0)\in\mathcal{A}(r_0)$, as the set $\mathcal{A}(r_0)$ is compact there is $b\in X$ such that, up to a subsequence, $\varphi_n(r_0)\to b$. From the LUUS property of $\mathscr{G}$ there exist $\varphi\in\mathscr{G}(r_0)$ such that $\varphi_n\to\varphi$ on compact subsets of $[r_0,\infty)$.

As $\{r_n\}\subset [r_0,t]$ there is $r'\in [r_0,t]$ such that, up to a subsequence, $r_n\to r'$. Given $\varepsilon>0$ and from the continuity of the solutions and uniform convergence of $\varphi_n$ in $[r_0,t]$, for $n$ large enough, we get 
\begin{align*}
\|\varphi_n(r_n)-\varphi(r')\|\leq \|\varphi_n(r_n)-\varphi(r_n)\|+\|\varphi(r_n)-\varphi(r')\|<\varepsilon.
\end{align*}

\

\noindent\textbf{Case 2:} $\inf_nr_n=-\infty$.

In this case, up to a subsequence, we may assume $r_n\searrow -\infty$. We have
\begin{align*}
dist(x_n,\mathcal{A}_{\infty})\leq dist(\mathcal{A}(r_n),\mathcal{A}_{\infty})\to 0.
\end{align*}
We can choose $y_n\in\mathcal{A}_{\infty}$ such that 
$$d(x_n,y_n)\leq dist(x_n,\mathcal{A}_{\infty})+\frac{1}{n}.$$
There is $y\in\mathcal{A}_{\infty}$ such that, up to a subsequence, $y_n\to y$, which implies that $x_n\to y$.
\end{proof}

%%%%%%%%%%%%%%%%%%%%%%%%%%%%%%%%%%%%%%%%%%%%%%%%%%%%%%%%%%%%%

\section{An application to a reaction-diffusion equation with dynamic boundary conditions}\label{application}

In \cite{spa} where the authors have introduced the concept of selected pullback attractor, as an application of the theory the authors have considered the following problem with a nonautonomous p-Laplacian equation, 
\begin{equation}\label{P}
\left\{
\begin{array}{ll}
u_t-\Delta_pu+ f_1(t,u)=g_1(t,x), & (t,x)\in
(\tau,+\infty)\times\Omega, \cr
u_t+|\nabla u|^{p-2}\deln{u}+f_2(t,u)=g_2(t,x), & \;(t,x)\in(\tau,+\infty)\times\Gamma, \cr
u(\tau)=u_0,
\end{array}
\right. \tag{$P$}
\end{equation}
where $\Omega\subset \mathbb{R}^N$ is a bounded domain with smooth boundary $\Gamma=\partial\Omega$, $N\geq 3$ and $\Delta_p$ denotes the $p$-Laplacian operator with $p\in(2,+\infty)$. One of the hypotheses for the perturbations $f_i$ are $f_i\in C(\R^2)$ and satisfies 
\begin{equation}\label{2.4}
a_i(t)|s|^{r_i}-k_i(t)\leq f_i(t,s)s, %\cr
\end{equation}
for a.a. $t\in\R$ and every $s\in\R$ with $a_i \in L^1_{loc} (\R)$  real functions such that $ a_i (t) \geq a_0>0 $ for some fixed real number $a_0 $, and $ k_i \in L^1_{loc} (\R) $ are positive functions, for $ i = 1,2$, besides that, there are functions $ C_i \in L^{\infty}_{loc} (\R) $, $ i = 1,2 $, such that $ |f_i (t, s)| \leq C_i(t)(|s|^{r_i-1} +1) $ a.e. for $t\in\R$ and each $s\in\R$. Some others hypotheses on the perturbations and external forces are required, see \cite{artpull} for details.

The initial state $u_0$ 
belongs to the space $\xdois$, where 
$$\mathbb{X}^2:=L^2(\Omega,dx)\times L^2(\Gamma,dS)=\{F=(f,g); f\in L^2(\Omega) \mbox{ and } g\in L^2(\Gamma)\},$$
with the norm 
$$\|F\|_{\mathbb{X}^2}=\left(\int_{\Omega}|f|^2 dx+\int_{\Gamma}|g|^2dS\right)^{\frac{1}{2}}.$$   
 This space can be identified with the space $L^2(\overline{\Omega},d\mu)$ where $d\mu=dx\oplus dS$, i.e., if $A\subset\overline{\Omega}$ is $\mu-\mbox{measurable}$, then $\mu(A)=|A\cap\Omega|+S(A\cap\Gamma)$ and $S$ is the surface measure in boundary $\Gamma$, see \cite{artpull} for more details.  
 
 The authors in \cite{artpull} could not ensure the uniqueness of solution with the assumed hypotheses and then they work with a possibility of the existence of others solutions. They have ensured the existence of a $\mathscr{D}$-pullback attractor for a generalized process compused only with solutions which are from Faedo-Galerkin method and attracts the families of sets $\{D(t): t\in\R\}$ of nonempty subsets of $L^{2r-2}(\Omega)\times L^{2r-2}(\Gamma)\subset L^2(\overline{\Omega},d\mu)$ such that
$$\lim_{s\rightarrow-\infty}e^{\theta s}[D(s)]=0,$$
where $[D(s)]=\sup\left\{\|u\|_{L^{2r-2}(\Omega)}^{2r-2}+\|v\|_{L^{2r-2}(\Gamma)}^{2r-2}:(u,v)\in D(s)\right\}$ where $r$ and $\theta>0$ are suitable constants, with $2r-2> 2$. 

The restriction on the generalized process is because of some technicalities to develop the estimates of the solution, and then the attraction is ensured only for solutions from Faedo-Galerkin method. As we do not have the uniqueness we have no guarantees that there exist another solution that it is not coming from a sequence of Faedo-Galerkin method.

 In particular, this $\mathscr{D}$-pullbak attractor when we consider a generalized process compused with any solution of the problem is a selected pullback attractor, it was observed in \cite{spa}.

First of all, consider all the assumptions given in section 5 of \cite{artpull}. Here, we will consider some additional assumptions on the perturbations of the operators and external forces, and ensure the continuity of the selected pullback attractor of the Problem \eqref{P}.  

\

\noindent\textbf{Assumption A:} For $\tilde{g}_1\in L^{2r-2}(\Omega)$ and $\tilde{g}_2\in L^{2r-2}(\Gamma)$ we have
$$\lim_{\tau\to+\infty}\int_{\tau}^{+\infty} \|g_1(\tau+s)-\tilde{g}_1\|^{2r-2}_{2r-2,\Omega}+\|g_2(\tau+s)-\tilde{g}_2\|^{2r-2}_{2r-2,\Gamma}ds=0.$$

\noindent\textbf{Assumption B:} 
\begin{equation}\label{H4.1}
\sup_{t\in\R}\int_{-\infty}^{t} e^{\theta(s-t)}(k_1(s)+k_2(s))ds<+\infty
\end{equation}
and
\begin{equation}\label{H4.2}
\sup_{t\in\R}\int_{-\infty}^{t} e^{\theta(s-t)}(k_1(s)^{r-1}+k_2(s)^{r-1})ds<+\infty.
\end{equation}
Besides that, $ C_i \in L^{\infty}(\R) $, for each $i\in\{1,2\}$.

\

Theorem~\ref{asdf} can naturally be applied for the Problem \eqref{P}. The next lemma will be necessary to obtain the asymptotic continuity, i.e., to be possible apply Theorem~\ref{qwer}.

\begin{lemma}\label{atralimit}
The Selected Pullback Attractor $\mathcal{A}=\{\mathcal{A}(t)\}_{t\in\R}$ of the Problem \eqref{P} is forward compact. 
\end{lemma}

\begin{proof}
Due to ensure the existence of $\mathscr{D}$-pullback attractor o the work \cite{artpull} the authors have showed the existence of a $\mathscr{D}$-pullback absorbing set, given by
\begin{equation}\label{B}
B:=\{B(t)\}_{t\in\R}:=\left\{B_{\vp}(0,R(t))\bigcap K(t)    \right \}_{t\in\R}.
\end{equation}
where $\vp$ is a Banach space compactly embedding in $\xdois$. What matters to us is that, assuming the Assumption B, we can write briefly  
\begin{align*}
R(t):=\bigg[C_0+C\bigg(e^{-\theta (t+1)}\int_{-\infty}^{t+1}e^{\theta s}(\|g_1(s)&\|^{2r-2}_{2r-2,\Omega}+\|g_2(s)\|^{2r-2}_{2r-2,\Gamma})ds\\ 
&+e^{-\theta t}\int_{-\infty}^{t+1}e^{\theta s}(\|g_1(s)\|^{2}_{2,\Omega}+\|g_2(s)\|^{2}_{2,\Gamma})ds\bigg)\bigg]^{\frac{1}{p}}. 
\end{align*}

From 
%\begin{align*}
%\|g_1(s)-\tilde{g}_1\|^{2}_{2,\Omega}+&\|g_2(s)-\tilde{g}_2\|^{2}_{2, \\
%&\leq \int_{-\infty}^{t+1}\|g_1(s)-\tilde{g}_1\|^{2r-2}_{2r-2,\Omega}+\|g_2(s)-\tilde{g}_2\|^{2r-2}_{2r-2,\Gamma}+Kds
%\end{align*}
 assumption A, there is $N\in\N$ such that 
\begin{align*}
&\int_{N}^{\infty}
\|g_1(s)-\tilde{g}_1\|^{2r-2}_{2r-2,\Omega}+\|g_2(s)-\tilde{g}_2\|^{2r-2}_{2r-2,\Gamma}ds<1.
\end{align*}
Note that, 
\begin{align*}
&\sup_{t\geq N}\int_{-\infty}^{t+1}e^{\theta (s-(t+1))}(\|g_1(s)\|^{2r-2}_{2r-2,\Omega}+\|g_2(s)\|^{2r-2}_{2r-2,\Gamma})ds \\
&\leq \sup_{t\geq N}\int_{-\infty}^{t+1}e^{\theta (s-(t+1))}(\|g_1(s)-\tilde{g}_1\|^{2r-2}_{2r-2,\Omega}+\|\tilde{g}_1\|^{2r-2}_{2r-2,\Omega}+\|g_2(s)-\tilde{g}_2\|^{2r-2}_{2r-2,\Gamma}+\|\tilde{g}_2\|^{2r-2}_{2r-2,\Gamma})ds \\
&\leq \sup_{t\geq N}\int_{-\infty}^{t+1}e^{\theta (s-(t+1))}(\|g_1(s)-\tilde{g}_1\|^{2r-2}_{2r-2,\Omega}+\|g_2(s)-\tilde{g}_2\|^{2r-2}_{2r-2,\Gamma})ds+\frac{\|\tilde{g}_1\|^{2r-2}_{2r-2,\Omega}}{\theta} + \frac{\|\tilde{g}_2\|^{2r-2}_{2r-2,\Gamma}}{\theta} \\
&\leq \sup_{t\geq N}\bigg(\int_{-\infty}^{N}e^{\theta (s-(t+1))}(\|g_1(s)-\tilde{g}_1\|^{2r-2}_{2r-2,\Omega})ds+\int_N^{t+1}\|g_1(s)-\tilde{g}_1\|^{2r-2}_{2r-2,\Omega}ds \\
& \quad\quad \quad \quad + \int_{-\infty}^{N} e^{\theta (s-(t+1))}(\|g_2(s)-\tilde{g}_2\|^{2r-2}_{2r-2,\Gamma})ds+\int_N^{t+1}\|g_2(s)-\tilde{g}_2\|^{2r-2}_{2r-2,\Gamma}\bigg) \\
&\hspace{5cm} +\frac{\|\tilde{g}_1\|^{2r-2}_{2r-2,\Omega}}{\theta} + \frac{\|\tilde{g}_2\|^{2r-2}_{2r-2,\Gamma}}{\theta}\\
&\leq \int_{-\infty}^{N}e^{\theta s}(\|g_1(s)-\tilde{g}_1\|^{2r-2}_{2r-2,\Omega}+\|g_2(s)-\tilde{g}_2\|^{2r-2}_{2r-2,\Gamma})ds +1+\frac{\|\tilde{g}_1\|^{2r-2}_{2r-2,\Omega}}{\theta} + \frac{\|\tilde{g}_2\|^{2r-2}_{2r-2,\Gamma}}{\theta},
\end{align*}
from the hypotheses for existence of $\mathscr{D}$-pullback attractor, see section 5 of \cite{artpull}, this last line of the inequality above is bounded.

From Lemma 2.1 of \cite{artpull}, we get  
\begin{align*}
\sup_{t\geq N}\int_{-\infty}^{t+1}&e^{\theta (s-(t+1))}(\|g_1(s)\|^{2}_{2,\Omega}+\|g_2(s)\|^{2}_{2,\Gamma})ds\\
&\leq \sup_{t\geq N}\int_{-\infty}^{t+1}e^{\theta (s-(t+1))}(\|g_1(s)\|^{2r-2}_{2r-2,\Omega}+\|g_2(s)\|^{2r-2}_{2r-2,\Gamma}+K)ds\\
&\leq \sup_{t\geq N}\int_{-\infty}^{t+1}e^{\theta (s-(t+1))}(\|g_1(s)\|^{2r-2}_{2r-2,\Omega}+\|g_2(s)\|^{2r-2}_{2r-2,\Gamma})ds+\frac{K}{\theta}<+\infty.
\end{align*}
And thus, 
\begin{align*}
\sup_{t\geq N} & \ \left(e^{-\theta t}\int_{-\infty}^{t+1}e^{\theta s}(\|g_1(s)\|^{2}_{2,\Omega}+\|g_2(s)\|^{2}_{2,\Gamma})ds\right) \\
&\leq e^{\theta}\left( \sup_{t\geq N}\int_{-\infty}^{t+1}e^{\theta (s-(t+1))}(\|g_1(s)\|^{2r-2}_{2r-2,\Omega}+\|g_2(s)\|^{2r-2}_{2r-2,\Gamma})ds+\frac{K}{\theta}\right)
\end{align*}
which in turn is bounded.

It means that $R(t)$ is uniformly bounded for $t\in\R$, and then
 $$\bigcup_{t\geq s}\mathcal{A}(t)\subset \bigcup_{t\geq s}B(t)\subset \bigcup_{t\geq s}B_{\vp}(0,R(t))$$
is bounded in $\vp$ for each $s\in\R$. Therefore, $\cup_{t\geq s}\mathcal{A}(t)$ is precompact in $\xdois$ for each $s\in\R$.  
\end{proof}

In the proof of the previous Lemma the $\mathscr{D}$-pullback absorbing set $B$ absorbs only the solutions that coming from Faedo-Galerking method, because of this it ensures just the existence of selected pullback attractor when we consider the generalized process $\mathscr{G}$ composed with all solutions of the Problem \eqref{P}. 

As an immediate consequence of Theorem~\ref{qwer} we obtain that 
 \begin{equation}
lim_{t\to\infty}dist(\mathcal{A}(t),\mathcal{A}(\infty))=0,
\end{equation}
where $\mathcal{A}(\infty):=\cap_{t\in\R}\overline{\cup_{r\geq t}\mathcal{A}(r)}$.

\subsection{The asymptotically autonomous case} An autonomous version of the Problem \eqref{P} were considered in \cite{galwarma} and \cite{gal}, where they considered the following problem  
\begin{equation}\label{Pa}
\left\{
\begin{array}{ll}
u_t-\Delta_pu+ \tilde{f}_1(u)=\tilde{g}_1(x), & (t,x)\in
(0,+\infty)\times\Omega, \cr
u_t+|\nabla u|^{p-2}\deln{u}+\tilde{f}_2(u)=\tilde{g}_2(x), & \;(t,x)\in(0,+\infty)\times\Gamma, \cr
u(0)=u_0,
\end{array}
\right. \tag{$Pa$}
\end{equation}
with a similar assumptions assumed in \cite{artpull}, but they assumed an additional assumption that the derivatives of the functions $\tilde{f}_i$ are bounded. This additional assumption ensure the uniqueness of solution. The authors of these works have studied the forward asymptotic behavior of solutions, therefore this problem posses a global attractor $\mathcal{A}_{\infty}$ in $\xdois$.  

Let us now suppose one more condition:

\

\noindent\textbf{Assumption C:} For each $\tau\in\R$ there exists a function $\alpha_{\tau}:[0,+\infty)\to [0,+\infty)$ such that $\alpha_{\tau}(t)\to 0$ as $\tau\to +\infty$ for each $t\in[0,+\infty)$ and 
$$\begin{array}{l}
\gl f_1((t+\tau),u(t+\tau))-\tilde{f}_1(v(t)),u(t+\tau)-v(t)\gr_{L^2(\Omega)}\geq -\alpha_{\tau}(t)\\
\mbox{and }\\
\gl f_2((t+\tau),\gamma(u)(t+\tau))-\tilde{f}_2(\gamma(v)(t)),\gamma(u)(t+\tau)-\gamma(v)(t)\gr_{L^2(\Gamma)}\geq -\alpha_{\tau}(t)
\end{array}$$
 for all  $t\in\R^+, \tau\in\R,$ and any solutions $U:=(u,\gamma(u))$ of the Problem \eqref{P} and $V:=(v,\gamma(v))$ of the Problem  \eqref{Pa}.

\begin{lemma}\label{convU} Suppose that assumptions \textbf{A} and \textbf{C} are satisfied. Then a solution $U=(u,\gamma(u))$ of the problem \eqref{P} with initial condition in $U_{\tau}\in\xdois$ converges to the solution $V=(v,\gamma(v))$ of problem \eqref{Pa} with initial condition in $V_0\in\xdois$, in the following sense:
$$\|U(\tau+t,\tau,U_{\tau})-V(t,V_0)\|_{\xdois}\to 0 \mbox{ as }\tau \to +\infty \mbox{ for each }t\geq 0,$$
whenever $\|U_{\tau}-V_0\|_{\xdois}\to 0$ as $\tau\to+\infty$.
\end{lemma}

Before the proof, note that in this Lemma $U(\tau+t,\tau,U_{\tau})$ represents any solution of the Problem \eqref{P}, even if this solution it is not coming from a Faedo-Galerkin method. Let us go to the proof of the previous Lemma.

\

\begin{proof} Let $U$ a solution of the Problem \eqref{P} and $V$ a solution of the Problem \eqref{Pa}. From the definition of weak solution of the Problem \eqref{P} multiplying by $U(t+\tau)-V(t)$, we have
\begin{equation}\label{fraca1}
\begin{array}{ll}
\gl \partial_tU(t+\tau) ,U-V \gr_{\xdois}  +\gl|\nabla u|^{p-2}\nabla u,\nabla u-v\gr_{L^2(\Omega)}\cr
\quad + \gl f_1(t+\tau,u),u(t+\tau)-v(t)\gr_{L^2(\Omega)}+\gl f_2(t+\tau,v),\gamma(u)(t+\tau)-\gamma(v)(t))\gr_{L^2(\Gamma)} \cr
\quad\quad\quad= \gl g_1(t+\tau),u(t+\tau)-v(t)\gr_{L^2(\Omega)}+\gl g_2(t+\tau),\gamma(u)(t+\tau)-\gamma(v)(t)\gr_{L^2(\Gamma)}
\end{array}
\end{equation}        
and for the Problem \eqref{Pa} with solution $V$, we get 
\begin{equation}\label{fraca2}
\begin{array}{ll}
\gl \partial_tV(t) ,U(t+\tau)-V(t) \gr_{\xdois}  +\gl|\nabla v|^{p-2}\nabla v,\nabla u-v\gr_{L^2(\Omega)}\cr
\quad\quad\quad + \gl \tilde{f}_1(v(t)),u(t+\tau)-v(t)\gr_{L^2(\Omega)}+\gl \tilde{f}_2(v(t)),\gamma(u)(t+\tau)-\gamma(v)(t))\gr_{L^2(\Gamma)} \cr
\quad\quad\quad\quad\quad\quad\quad= \gl \tilde{g}_1,u(t+\tau)-v(t)\gr_{L^2(\Omega)}+\gl \tilde{g}_2,\gamma(u)(t+\tau)-\gamma(v)(t)\gr_{L^2(\Gamma)}
\end{array}
\end{equation}        
subtracting \eqref{fraca1} of \eqref{fraca2},
 \begin{equation*}\label{fraca3}
\begin{array}{ll}
\frac{1}{2}\frac{d}{dt}\|U(t+\tau)-V(t)\|^2_{\xdois}  +\gl|\nabla u|^{p-2}\nabla u-|\nabla v|^{p-2}\nabla v,\nabla u-v\gr_{L^2(\Omega)}\cr
\quad\quad\quad + \gl f_1(t+\tau,u(t+\tau))-\tilde{f}_1(v(t)),u(t+\tau)-v(t)\gr_{L^2(\Omega)} \cr
\quad\quad \quad \quad +\gl f_2((t+\tau),u(t+\tau))-\tilde{f}_2(v(t)),\gamma(u)(t+\tau)-\gamma(v)(t))\gr_{L^2(\Gamma)} \cr
\quad\quad\quad\quad\quad\quad\quad= \gl g_1(t+\tau)-\tilde{g}_1,u(t+\tau)-v(t)\gr_{L^2(\Omega)} \cr
\quad\quad\quad\quad\quad\quad\quad\quad +\gl g_2(t+\tau)-\tilde{g}_2,\gamma(u)(t+\tau)-\gamma(v)(t)\gr_{L^2(\Gamma)}
\end{array}
\end{equation*}

From Tartar inequality, assumption C and Young's Inequality
$$\frac{d}{dt}\|U(t+\tau)-V(t)\|^2_{\xdois}\leq 4\alpha_{\tau}(t)+c\|U(t+\tau)-V(t)\|^2_{\xdois}+\|g_1(t+\tau)-\tilde{g}_1\|^2_{2,\Omega}+\|g_2(t+\tau)-\tilde{g}_2\|^2_{2,\Gamma}$$
The Gronwall Lemma gives 
\begin{align*}
\|U(t+\tau)-V(t)\|_{\xdois}&\leq e^{ct}\|U_{\tau}-V_0\|_{\xdois}\\
&\quad +\int_0^te^{c(t-s)}(4\alpha_{\tau}(s)+\|g_1(s+\tau)-\tilde{g}_1\|^2_{2,\Omega}+\|g_2(s+\tau)-\tilde{g}_2\|^2_{2,\Gamma})ds \\
&\leq e^{ct}\bigg(\|U_{\tau}-V_0\|_{\xdois}+\frac{4t}{c}esssup_{s\in[0,+\infty)}\alpha_{\tau}(s) \\
&\quad \quad + \int_0^{+\infty} e^{-cs}(\|g_1(s+\tau)-\tilde{g}_1\|^2_{2,\Omega}+\|g_2(s+\tau)-\tilde{g}_2\|^2_{2,\Gamma})ds\bigg)
\end{align*} 
and, from Lemma 2.1 of \cite{artpull}, we have 
\begin{align*}
\int_0^{+\infty} e^{-cs}&(\|g_1(s+\tau)-\tilde{g}_1\|^2_{2,\Omega}+\|g_2(s+\tau)-\tilde{g}_2\|^2_{2,\Gamma})ds \\
&\leq \int_{\tau}^{+\infty} (\|g_1(s)-\tilde{g}_1\|^2_{2,\Omega}+\|g_2(s)-\tilde{g}_2\|^2_{2,\Gamma})ds \\
&\leq  \int_{\tau}^{+\infty} \|g_1(s)-\tilde{g}_1\|^{2r-2}_{2r-2,\Omega}+\|g_2(s)-\tilde{g}_2\|^{2r-2}_{2r-2,\Gamma}+K ds
\end{align*}
and 
\begin{equation*}\label{limitK}
\begin{split}
\lim_{\tau\to \infty}\int_{\tau}^{\infty}Kds=\lim_{\tau\to\infty}\left(\lim_{a\to\infty}Ka-K\tau\right)=\lim_{a\to\infty}Ka-\lim_{\tau\to\infty}K\tau=\lim_{a\to\infty}Ka-Ka=0.
\end{split}
\end{equation*}

Therefore, when $\tau\to +\infty$ we have that 
\begin{align*}
&\|U_{\tau}-V_0\|_{\xdois} \to 0 \hspace{2cm}\mbox{and} \hspace{2cm}\frac{4t}{c}ess\sup_{s\in[0,+\infty)}\alpha_{\tau}(s)\to 0.
\end{align*}
and, from assumption A, we have that $U(t+\tau)$ converges to $V(t)$ in $\xdois$ when $\tau\to\infty$ for each $t\in\R^+$. 
\end{proof}

From Lemma \ref{convU} we have that the generalized process $\mathscr{G}$ is asymptotically autonomous for $\mathcal{A}$, and from Lemma \ref{atralimit}  we have that $\mathcal{A}$ is forward compact. Therefore the conditions of Theorem \ref{TeoX} are satisfied, and then we can ensure the following result.

\begin{theorem} Suppose the Assumptions A, B and C are satisfied. Then 
$$\lim_{t\to \infty}dist(\mathcal{A}(t),\mathcal{A}_{\infty})=0.$$
\end{theorem}

%%%%%%%%%%%%%%%%%%%%%%%%%%%%%%%%%%%%%%%%%%%

\

%\textbf{Acknoledgments.} 

\end{document}